\documentclass[10pt]{article}

\newtheorem{theorem}{Theorem}[section]

\newtheorem{corollary}[theorem]{Corollary}
\newtheorem{definition}[theorem]{Definition}
\newtheorem{proposition}[theorem]{Proposition}

\newtheorem{lemma}[theorem]{Lemma}

\newenvironment{proof}{{\sc Proof.}}{\medskip}
\newcommand{\qed}{\qquad $\Box$}

\usepackage{ahthesis}

\def\AA{A}

\renewcommand{\pleq}{\leq}

\begin{document}

\title{Supercharacter theories and Schur rings}
\author{Anders O. F. Hendrickson\\
        {\em Concordia College, Moorhead, Minnesota, USA}}
\date{June 7, 2010}

\def\sidenote#1{\marginpar{\tiny\raggedright #1}}

\maketitle

\begin{abstract}
  Diaconis and Isaacs have defined the supercharacter theories 
  of a finite group
  to be certain approximations to the ordinary character theory of the group.
  We make explicit the connection between supercharacter theories
  and Schur rings,
  and we provide supercharacter theory constructions which correspond
  to Schur ring constructions of Leung and Man and of Tamaschke.
\end{abstract}

\section{Introduction}

Supercharacter theories of a finite group 
were defined by Diaconis and Isaacs 
\cite{diaconis_isaacs} as approximations
to the group's ordinary character theory.
In a supercharacter theory
certain (generally reducible) characters take the place of irreducible characters,
and the role of conjugacy classes is played by certain unions of conjugacy classes.
For the group $U_n(\F_q)$ 
of upper triangular matrices over the field of size $q$
with all diagonal entries one \cite{andre1995, andre1999, andre2002, yan},
and more generally for 
groups of the form $1+\mathfrak{n}$ where $\mathfrak{n}$ is a nilpotent associative $\F_q$-algebra \cite{diaconis_isaacs},
a particularly nice supercharacter theory exists which is 
simple enough to be computed explicitly,
yet still rich enough to handle some problems
that traditionally required knowing the full character theory
\cite{arias2004}.
More recent interest has turned to the relationship of these supercharacters with 
the Hopf algebra of symmetric functions of noncommuting variables \cite{thiem_branching}.

As it turns out, there is a strong connection between the supercharacter theories defined in \cite{diaconis_isaacs}
and the Schur rings of a finite group.  
In fact, several of the initial theorems in \cite{diaconis_isaacs}
appeared in a different form in the work of Tamaschke on Schur rings \cite{tamaschke1970}.
Related work is also being done by Humphries and Johnson \cite{humphries_johnson2008},
who ask (in the language of \cite{diaconis_isaacs}) 
which groups have a character table that is identical with a supercharacter tables of some abelian group.
In this paper we formalize the connection between supercharacter theories and Schur rings
and give constructions of supercharacter theories
corresponding to the Schur ring wedge product of Leung and Man \cite{leung_man1996}
and the tensor product of Tamaschke \cite{tamaschke1970}.

\section{Correspondence}

\begin{definition}\label{defn_supercharthy}
  Let $G$ be a finite group,
  let $\KK$ be a partition of $G$,
  and let $\XX$ be a partition of the set $\Irr(G)$.
  Suppose that for every part $X\in\XX$ there exists a character $\chi_{_X}$
  whose irreducible constituents lie in $X$,
  and suppose the following three conditions hold.
  \begin{enumerate}
    \item Each of the characters $\chi_{_X}$
          is constant on every part $K\in\KK$.
    \item $|\XX|=|\KK|$.
    \item Every irreducible character is a constituent of some $\chi_X$.
  \end{enumerate}
  Then 
  we call the characters $\chi_{_X}$ \defnstyle{supercharacters},
  we call the members of $\KK$ \defnstyle{superclasses},
  and we say that the ordered pair $(\XX,\KK)$ 
  is a \defnstyle{supercharacter theory}.
  If $\CCC=(\XX,\KK)$ is a supercharacter theory,
  then we define $|\CCC|$ to be the integer equal to
  both $|\XX|$ and $|\KK|$.
  We write $\Sup(G)$ for the set of all supercharacter theories of $G$.
\end{definition}

Assume that $(\XX,\KK)$ is a supercharacter theory of a group $G$,
and for every subset $X$ of $\Irr(G)$ let $\s_X$ be the character $\sum_{\psi\in X}\psi(1)\psi$.
Diaconis and Isaacs
prove in \cite[Lemma 2.1]{diaconis_isaacs}
that $\{1\}\in\KK$, that $\{1_G\}\in\XX$,
and that for each $X\in\XX$, 
the supercharacter $\chi_{_X}$ must be a constant multiple of $\s_X$.
It is therefore no loss to assume that $\chi_{_X}=\s_X$,
and we shall make that assumption throughout this paper.
It is also shown in 
\cite[Theorem 2.2(c)]{diaconis_isaacs}
that if $\CCC=(\XX,\KK)$ is a supercharacter theory,
then each of the partitions $\XX$ and $\KK$
uniquely determines the other.

The concept of a Schur ring is significantly older,
having been defined by Schur \cite{schur_srings} and developed further by Wielandt \cite{wielandt}.
In the following definition, for each subset $K$ of a group $G$,
we let $\hat{K}=\sum_{g\in K} g$ in the group algebra $\C[G]$.

\begin{definition}\label{defn_sring}
  Let $G$ be a finite group.
  A subring $\AA$ of the group algebra $\C[G]$ is called a \defnstyle{Schur ring} or \defnstyle{S-ring} over $G$
  if there exists a set partition $\KK$ of $G$ satisfying the following conditions:
  \begin{enumerate}
    \item The set $\hat{\KK}=\{\hat{K}: K\in\KK\}$ is a linear basis of $\AA$,
    \item $\{1\}\in\KK$, and
    \item $\{g\inv: g\in K\}\in\KK$ for all $K\in\KK$. \label{xinvcondn}
  \end{enumerate}
  In this case we say the partition $\KK$ is a \defnstyle{Schur partition}
  and call its parts the \defnstyle{basic sets} of $\AA$.
  We will usually write $\AA=S_\KK$.
\end{definition}

Note that the partition $\KK$ is completely determined by the ring $\AA$.
S-rings were originally used by Schur and Wielandt in the study of permutation groups \cite[pp.~403--412]{scott},
but recently they have found applications in algebraic combinatorics,
especially in the study of circulant graphs.
A good survey of recent developments in S-rings can be found in \cite{muzychuk_ponomarenko}.

We next establish the correspondence between supercharacter theories and certain S-rings,
making use of the following lemma.
  Recall that every character $\chi\in\Irr(G)$
  has a corresponding central idempotent
  $e_\chi = \rec{|G|}\chi(1)\sum_{g\in G} \overline{\chi(g)}\,g$.
For a subset $X\sseq\Irr(G)$, let $f_X=\sum_{\psi\in X} e_\psi$.

\begin{lemma}\label{lem_subalgyieldsfZs}
  Let $G$ be a group and let $\AA$ be a subalgebra of $\Zb(\C[G])$.
  Then there exists a unique partition $\XX$ of $\Irr(G)$
  such that $\{f_X: X\in\XX\}$ is a basis for $\AA$.
\end{lemma}
\begin{proof}
  Recall that the set $\{e_\chi: \chi\in\Irr(G)\}$ is a basis for $\Zb(\C[G])$.
  Thus $\Zb(\C[G])$ is isomorphic to a direct sum of copies of $\C$,
  so it contains no nilpotent elements,
  and neither does its subalgebra $\AA$;
  hence the Jacobson radical $\Jb(\AA)=0$.
  Then by Wedderburn's theorem
  $\AA$ is a direct sum of full matrix rings;
  but since $\AA$ is commutative, 
  those are rings of $1\by 1$ matrices,
  so $\AA$ too is a direct sum of copies of $\C$.
  Hence $\AA$ is the linear span of some idempotents $f_1,\ldots,f_r$
  whose sum is 1 and whose pairwise products are 0.
  But every idempotent in $\Zb(\C[G])$ is a sum of some distinct $e_\chi$,
  and because 
  $\sum_{i=1}^r f_i = 1 = \sum_{{\chi\in\Irr(G)}} e_\chi$ 
  but the product $f_i f_j=0$ for $i\neq j$, 
  every $e_\chi$ must appear in exactly one $f_i$.
  Thus there exists a partition $\XX$ of $\Irr(G)$ such that 
  $\{f_X: X\in\XX\}=\{f_1,\ldots,f_r\}$,
  and this is the desired basis for $\AA$.

  To show uniqueness, suppose $\YY$ is also a partition of $\Irr(G)$
  such that 
  $$\spanof{f_Y: Y\in\YY}=\AA=\spanof{f_X:X\in\XX}.$$
  Let $\chi\in\Irr(G)$, and let $Y_0\in\YY$ and $X_0\in\XX$ 
  be the parts containing $\chi$.
  Then because $f_{Y_0}\in\AA=\spanof{f_X:X\in\XX}$, 
  the set $Y_0$ must be a union of parts of $\XX$;
  in particular $X_0\sseq Y_0$.
  But by symmetry $Y_0\sseq X_0$, so $Y_0=X_0$.
  Since the parts of $\YY$ and $\XX$ containing $\chi$ are identical
  for all $\chi\in\Irr(G)$, it follows that $\YY=\XX$.\qed
\end{proof}

\begin{proposition}\label{sringbijection}
  Let $G$ be a finite group.
  Then there is a bijection
  $$\begin{array}{rccc}
           & \left\{\parbox{1.75in}{\raggedright supercharacter theories $(\XX,\KK)$ of $G$}\right\}
           &\longleftrightarrow&
           \left\{\parbox{1.5in}{S-rings of $G$ \\ contained in $\Zb(\C[G])$}\right\} \\
      & (\XX,\KK) & \longmapsto & \spn_\C\{\Khat: K\in\KK\}.
    \end{array}$$
\end{proposition}

\begin{proof}
  Let $(\XX,\KK)\in\Sup(G)$,
  and let $A$ be the subspace spanned by $\{\Khat: K\in\KK\}$.
  By \cite[Theorem 2.2(b)]{diaconis_isaacs}, we know $A$ is a subalgebra of $\Zb(\C[G])$.
  Now the set $\{1\}$ is one of the superclasses,
  and 
  the map $g\mapsto g\inv$ permutes the superclasses
  by \cite[Theorem 2.2(f)]{diaconis_isaacs};
  thus $A$ is an S-ring.  
  
  As the partition $\KK$ can be recovered
  from $A=\spn_\C\{\Khat: K\in\KK\}$
  as the set $\{ K\sseq G: \Khat\in A\mbox{~but~} \hat{L}\not\in A \mbox{~for all~} L\subsetneq K\}$,
  the map $(\XX,\KK)\mapsto \spn_\C\{\Khat: K \in\KK\}$
  is injective.
  It thus remains to show that every S-ring contained in $\Zb(\C[G])$
  corresponds to a supercharacter theory.

  Let $A$ be an S-ring contained in $\Zb(\C[G])$.
  Then by Lemma \ref{lem_subalgyieldsfZs},
  there exists a partition $\XX$ of $\Irr(G)$ such that 
  $\{f_X: X\in\XX\}$ is a basis for $A$,
  where $f_X = \sum_{\chi\in X} e_\chi$.
  Now let $\KK$ be the partition of $G$ into the basic sets of $A$;
  we shall show that $(\XX,\KK)\in\Sup(G)$.
  Since $\{f_X: X\in\XX\}$ and $\{\Khat: K\in\KK\}$
  are both bases of $A$, we have that $|\XX|=\dim A = |\KK|$,
  and it only remains to show that $\sigma_X$ is constant on $K$
  for all $X\in\XX$ and all $K\in\KK$.

  But because $e_\chi = \rec{|G|}\sum_{g\in G}\overline{\s_{\{\chi\}}(g)}\,g$,
  by the linearity of the $\sigma$ operator we have
  \begin{equation*}\label{eqn_fXsXreln}
  f_X   = \summ_{\chi\in X} e_\chi
        = \rec{|G|} \sumg \bar{\s_X(g)} \, g .
  \end{equation*}
  Then because $f_X\in\AA$, 
  the function $g\mapsto \rec{|G|}\bar{\s_X(g)}$ 
  is constant on all $K\in\KK$,
  so $\s_X$ must be constant on all $K\in\KK$ as well.
  Therefore $(\XX,\KK)$ is a supercharacter theory of $G$.\qed
\end{proof}

So the supercharacter theory $(\XX,\KK)$ corresponds to the S-ring $S_\KK$.
It is worth noting that our proof of surjectivity
did not make use of condition (\ref{xinvcondn}) from the definition of S-rings;
hence any subring $A$ of $\Zb(\C[G])$ 
satisfying conditions (a) and (b) of Definition \ref{defn_sring} 
corresponds to a supercharacter theory,
and it follows that $A$ is in fact an S-ring.
Note that requiring an S-ring to lie in the center of the group algebra 
is the same as requiring all its basic sets to be unions of conjugacy classes.
For this reason, 
often supercharacter constructions often require normality
where the corresponding S-ring constructions do not.
We may also note that the supercharacter theories of an abelian group 
are in one-to-one correspondence with its S-rings.

\section{Lattice theory}
It is known that the S-rings of a group $G$ form a 
lattice with respect to inclusion \cite[Section 2]{muzychuk_ponomarenko}.
The set of supercharacter theories of a group also form a lattice
in the following natural way.
Recall that the set $\Part(S)$ of all partitions of a set $S$ forms a lattice,
in which $\KK \leq \LL$ iff every part of $\KK$ is a subset of some part of $\LL$.
We could therefore make $\Sup(G)$ into a poset
by defining that $(\XX,\KK)\leq (\YY,\LL)$ if $\XX\leq \YY$,
or we could equally well define that $(\XX,\KK)\leq(\YY,\LL)$ if $\KK\leq\LL$.
The purpose of this section is to show that these two definitions are equivalent;
along the way we will explicitly find 
the lattice-theoretic join of two supercharacter theories.

For each partition $\XX$ of $\Irr(G)$, let $\Ab_\XX=\spn\{f_X: X\in\XX\}$.
\begin{lemma}\label{lem_fGalois}
The map $\XX\mapsto \Ab_\XX$ is a bijection 
from $\Part(\Irr(G))$ to the set of subalgebras of $\Zb(\C[G])$.
This map is order-reversing, in the sense that
for all $\XX,\YY\in\Part(\Irr(G))$,
we have $\XX\pleq\YY$ if and only if $\Ab_\YY\sseq \Ab_\XX$.\label{part_fGalois_po}
In particular, $\Ab_{\XX\join\YY} = \Ab_\XX \cap \Ab_\YY$. \label{part_fGalois_XcupY}
\end{lemma}

\begin{proof}
  Lemma \ref{lem_subalgyieldsfZs} shows that the map $\XX\mapsto\Ab_\XX$ is invertible.
  Now suppose $\XX\pleq\YY$ and let $Y$ be a part of $\YY$.
  Then $Y$ is the union of some parts of $\XX$,
  so $f_Y$ is the sum of the corresponding idempotents $f_X$.
  Then $f_Y\in\Ab_\XX$; hence $\Ab_\YY\sseq \Ab_\XX$.
  Conversely, suppose $\Ab_\YY\sseq \Ab_\XX$.
  Let $Y$ be a part of $\YY$.
  Then $f_Y$ is an idempotent in $\Ab_\XX$,
  so it is a sum of some of the spanning idempotents $\{f_X: X\in\XX\}$ of $\Ab_\XX$.
  It follows that $Y$ must be a union of parts of $\XX$.
  Thus the map $\XX\mapsto \Ab_\XX$ is an order-reversing bijection, as desired.

  Then because $\XX\join\YY$ is the least upper bound for $\XX$ and $\YY$ in $\Part(\Irr(G))$,
  the subalgebra $\Ab_{\XX\join\YY}$ must be the largest subalgebra
  contained in both $\Ab_\XX$ and $\Ab_\YY$, namely $\Ab_\XX \cap \Ab_\YY.$\qed
  %
\end{proof}

Let $\KK,\LL\in\Part(S)$ for some set $S$.
Note that if $f$ is a function whose domain is $S$
that is constant on each part of $\KK$ and constant on each part of $\LL$,
then $f$ must be constant on each part of the partition $\KK\join\LL$.

\begin{lemma}\label{lem_KhalfGalois}
Let $\KK,\LL\in\Part(G)$.
Then
\begin{equation}\label{eqn_intersectKspans}
\spanof{\Mhat: M\in\KK\join\LL} = \spanof{\Khat: K\in\KK} \cap \spanof{\Lhat: L\in\LL}.
\end{equation}
\end{lemma}
\begin{proof}
Since each part $M\in\KK\join\LL$ 
is a union of some parts of $\KK$, 
the sum $\Mhat$ lies in $\textspanof{\Khat: K\in\KK}$.
Likewise $\Mhat\in\textspanof{\Lhat: L\in\LL}$,
so the left side of (\ref{eqn_intersectKspans}) is contained in the right hand side.

On the other hand, each element $d$ on the right side of (\ref{eqn_intersectKspans})
may be written as 
$d = \sum_{K\in\KK} a_K \Khat = \sum_{L\in\LL} b_L \Lhat$ for some coefficients $a_K,b_L\in\C$.
Recall that each element $g\in G$ occurs in exactly one $K$ and in exactly one $L$,
and that $G$ is a basis for $\C[G]$.
Now the function mapping $g$ to the coefficient of $g$ in $d$ is constant on
each $K\in\KK$, and also constant on each $L\in\LL$.
Hence it is constant on each member of $\KK\join\LL$ 
and so $d$ lies in the span of $\{\Mhat: M\in\KK\join\LL\}$. \qed
\end{proof}

These two lemmas allow us to define a lattice-theoretic join operation on
supercharacter theories of a group $G$:

\begin{proposition}\label{prop_joinisthy}
  Let $G$ be a group.
  Let 
  $(\XX, \KK)$ and $(\YY,\LL)$ be supercharacter theories of $G$.
  Then $(\XX\join \YY\, ,\, \KK\join \LL)$ is also a supercharacter theory of $G$,
  which we denote $(\XX,\KK)\join(\YY,\LL)$.
\end{proposition}
\begin{proof}
Let $\ZZ = \XX\join\YY$ and $\MM=\KK\join\LL$.
%
To show that the functions $\{\sigma_Z: Z\in\ZZ\}$
are constant on the sets $M\in\MM$,
let $Z\in\ZZ$, let $M\in\MM$, and let $g,h\in M$.
Now $Z={\bigcup}_{X\in \II} X$ for some subset $\II\sseq\XX$,
so $\s_Z = \sum_{X\in \II} \s_X$ must be constant on each $K\in\KK$
because $(\XX,\KK)$ is a supercharacter theory.
On the other hand, by symmetry
$\s_Z$ is also constant on each $L\in\LL$.
So 
it follows that $\s_Z$ is constant on each $M\in\MM$.
It only remains to show that $|\ZZ|=|\MM|$.

Recall 
that $\left\{f_X: X \in \XX\right\}$ and $\{\Khat: K\in\KK\}$ 
are two different bases for the same algebra $S_\KK$,
and likewise 
$\left\{f_Y: Y\in\YY\right\}$ and $\{\hat{L}: L\in\LL\}$
are bases for $S_\LL$.
Hence
$$\begin{array}{r@{}c@{}l@{}l}
  \multicolumn{3}{l}{\spanof{f_Z: Z\in\ZZ}} \\
  \quad &{}={}& \spanof{f_X: X\in\XX} \cap \spanof{f_Y: Y\in\YY}  
                                  & \mbox{~(by Lemma \ref{part_fGalois_XcupY})} \\
                      &=& \spanof{\Khat: K\in\KK} \cap \spanof{\Lhat: L\in\LL} \\
                      &=& \spanof{\Mhat: M\in\MM}.
                                  & \mbox{~(by Lemma \ref{lem_KhalfGalois})} \\
  \end{array}
  $$
But since both $\{f_Z: Z\in\ZZ\}$ and
$\{\Mhat: M\in\MM\}$ are linearly independent sets over $\C$, 
both $|\ZZ|$ and $|\MM|$ must equal the dimension of the algebra in question,
so $|\ZZ|=|\MM|$ as desired.
We conclude that $(\XX\join\YY\widecomma \KK\join\LL)$ is a supercharacter theory of $G$.\qed
\end{proof}

Taking the join of two supercharacter theory thus corresponds to intersecting their S-rings.
On the other hand, the partition meets $\XX\meet\YY$ and $\KK\meet\LL$ in general do not form a supercharacter theory;
this corresponds to the fact that the subalgebra generated by two S-rings need not itself be a S-ring.

\begin{corollary}\label{prop_XlatticeisKlattice}
Let $(\XX,\KK)$ and $(\YY,\LL)$ be supercharacter theories of $G$.
Then $\XX\pleq \YY$ if and only if $\KK\pleq\LL$.
\end{corollary}
\begin{proof}
$(\Leftarrow)$
Suppose $\KK\pleq\LL$.
Since $(\XX,\KK)$ and $(\YY,\LL)$ are supercharacter
theories for $G$, so is 
$(\XX\join\YY \widecomma \KK\join\LL)$, which is equal to $(\XX\join\YY \widecomma \LL)$.
But because superclasses and supercharacters determine one another, 
there is only one supercharacter theory with superclasses $\LL$,
namely $(\YY,\LL)$.  So
$\XX\join\YY = \YY$ and $\XX\pleq\YY$ as desired.

$(\Rightarrow)$
Suppose $\XX\pleq\YY$.
Then $(\XX\join\YY \widecomma \KK\join\LL) = (\YY\widecomma \KK\join\LL)$
is a supercharacter theory for $G$.
Because $(\YY,\LL)$ is the unique supercharacter theory with supercharacters from $\YY$,
we must have $\KK\join\LL=\LL$
and so $\KK\pleq\LL$.
\qed
\end{proof}

We are therefore not breaking symmetry between superclasses and supercharacters
when we define a partial ordering of $\Sup(G)$ as follows. 

\begin{definition}
  Let $(\XX,\KK)$ and $(\YY,\LL)$ be supercharacter theories of a group $G$.
  Then we write $(\XX,\KK)\pleq(\YY,\LL)$ if $\XX\pleq\YY$.
\end{definition}

Since $\Sup(G)$ may be viewed as a subset of the finite lattice $\Part(\KK)$,
it is in fact a lattice itself.
Whereas computing $\CCC\join\DDD$ in this lattice can be done by taking partition joins,
in general $\CCC\meet\DDD$ is not readily computable.
Note that the set of S-rings of $G$ lying in $\Zb(\C[G])$,
which is a sublattice of the lattice of S-rings,
is isomorphic to the dual of the supercharacter theory lattice $\Sup(G)$.

\section{$*$-products}

If $N$ is a  normal subgroup of $G$,
then some supercharacter theories of $N$ 
can be combined with
supercharacter theories of $G/N$
to form supercharacter theories of the full group $G$.
We can thus construct supercharacter theories of large groups 
by combining those of smaller groups.
This is the most important special case of 
a more general construction
which will be treated in Section \ref{sect_wtp}.

Let a group $G$ act on another group $H$;
then as discussed in \cite{diaconis_isaacs}, 
there exists a supercharacter theory
$\Conj{G}{H}=(\XX,\KK)\in\Sup(H)$ 
such that $\XX$ is the partition of $\Irr(H)$ into $G$-orbits
and $\KK$ is the finest partition of $H$ into unions of conjugacy classes
such that each part is $G$-invariant.
Then for another supercharacter theory $(\YY,\LL)\in\Sup(H)$,
every part $L\in\LL$ is $G$-invariant if and only if $\KK\pleq\LL$,
which is true if and only if $\XX\pleq\YY$ (by Corollary \ref{prop_XlatticeisKlattice}), 
which is true
if and only if every part $Y\in\YY$ is $G$-invariant.
We may thus unambiguously speak of the \defnstyle{$G$-invariant supercharacter theories} of $H$,
denoted $\Supinv{G}(H)$,
of which $\Conj{G}{H}$ 
is the minimal one.

%
%

We would like to define a product
$$*: \Supinv{G}(N)\by\Sup(G/N)\longrightarrow \Sup(G).$$
So suppose $\CCC=(\XX,\KK)\in\Supinv{G}(N)$ and $\DDD=(\YY,\LL)\in\Sup(G/N)$.
Let us first consider the superclasses: 
$\KK$ is a $G$-invariant partition of $N$
and $\LL$ a $G$-invariant partition of $G/N$, 
one part of which is the coset $N$ as a singleton.
Thus $\LL$ induces a partition $\tw{\LL}$ of $G$,
one part of which will be the set $N$,
which we can then replace with the partition $\KK$ of $N$.
To express this formally, 
for each subset $L\sseq G/N$ let $\tw{L}={\bigcup}_{Ng\in L}Ng$,
and extend this notation to $\LL$ 
by $\tw{\LL}=\{\tw{L}: L\in\LL\}$.
Then for $\KK\in\Part(N)$ and $\LL\in\Part(G/N)$, we have a partition
\begin{equation}\label{eqn_classpart}
  \KK\cup\tw{\LL}-\{N\}\in\Part(G).
\end{equation}


For the supercharacters, 
if $N$ is a normal subgroup of $G$ and $\psi\in\Irr(N)$,
let $\Irr(G|\psi)$ denote $\{\chi\in\Irr(G) : [\chi_N,\psi]>0\}$.
If $Z\sseq\Irr(N)$ is a union of $G$-orbits,
then define the subset $Z^G$ of $\Irr(G)$
to be
${\bigcup}_{\psi\in Z} \Irr(G|\psi)$.
Extend this notation to a set $\ZZ$ of such subsets of $\Irr(N)$
by letting $\ZZ^G=\{Z^G: Z\in\ZZ\}$.

Now consider $(\XX,\KK)\in\Sup_G(N)$ and $(\YY,\LL)\in\Sup(G/N)$ as before.
Since $\XX$ is a partition of $\Irr(N)$ into unions of $G$-orbits,
it follows from Clifford theory that $\XX^G$ is a partition of $\Irr(G)$.
Since $\{1_N\}\in\XX$, one part of $\XX^G$ is 
$\{1_N\}^G=\{\chi\in\Irr(G) : N\sseq\ker\chi\}$, 
which we identify with $\Irr(G/N)$ in the usual way.
Thus we can replace that part of $\XX^G$
with the partition $\YY$ of $\Irr(G/N)$,
obtaining a partition of $G$
\begin{equation}\label{eqn_charpart}
  \YY\cup \XX^G - \{\Irr(G/N)\}\in\Part(\Irr(G))
\end{equation}
incorporating information from both $\XX$ and $\YY$.


We shall show that the partitions of (\ref{eqn_classpart}) and (\ref{eqn_charpart})
do form a supercharacter theory of $G$,
by way of a brief lemma
demonstrating the suitability of the notation ``$X^G$.''

\begin{lemma}\label{lem_calculatesigmaXG}
Let $N\norm G$ and let $X\sseq \Irr(N)$
be a union of $G$-orbits.
Then 
$$\sigma_{(X^G)} = \left(\sigma_X\right)^G.$$
\end{lemma}
\begin{proof}
  It suffices to prove the statement when $X$ is a single $G$-orbit.
  Let $\ZZ$ be the partition of $\Irr(N)$ into its $G$-orbits,
  so that $\ZZ^G$ is a partition of $\Irr(G)$.
  Then the regular character 
  $\rho_N=\sum_{{Z\in\ZZ}} \s_Z$ 
  and so
  $$\rho_G = (\rho_N)^G=\sum_{Z\in\ZZ} (\s_Z)^G.$$
  Now the characters $(\s_Z)^G$ have no irreducible constituents in common
  with one another, 
  so $(\s_X)^G=\s_Y$ where $Y$ is the set of irreducible constituents of $(\s_X)^G$.
  But $Y=X^G$, so we conclude that
  $\s_{(X^G)} = (\s_X)^G$.\qed
\end{proof}


\begin{theorem}\label{thm_starproduct}
  Let $G$ be a group and let $N\norm G$.
  Suppose $(\XX,\KK)\in\Supinv{G}(N)$ and suppose $(\YY,\LL)\in\Sup(G/N)$.
  Then
  $$\left(
       \YY\cup \XX^G - \{\Irr(G/N)\},~
       \KK\cup\tw{\LL}-\{N\}
    \right)$$
  is a supercharacter theory of $G$.
\end{theorem}
\begin{proof}
  Let $\ZZ=\YY\cup \XX^G-\{\Irr(G/N)\}\in\Part(\Irr(G))$
  and let $\MM=\KK\cup\tw{\LL}-\{N\}\in\Part(G)$.
  Now
  $|\ZZ|=|\YY|+|\XX|-1 = |\LL|+|\KK|-1 = |\MM|$,
  so it remains to show that for each part $Z\in\ZZ$,
  the character $\sigma_Z$ is constant on each subset $M\in\MM$.

  One possibility is that $Z$ lies in $\YY$.
  In this case, because $Z\sseq\Irr(G/N)$,
  the character $\sigma_Z$ has $N$ in its kernel,
  so it is certainly constant on every part $K\in\KK$.
  Moreover, because $(\YY,\LL)$ is a supercharacter theory of $G/N$,
  we know that $\sigma_Z$ (viewed as a character of $G/N$)
  is constant on each superclass $L\in\LL$;
  viewed as a character of $G$, 
  it is therefore constant on each $\tw{L}\in\tw{\LL}$.
  So $\sigma_Z$ is constant on each set $M\in\MM$ in the case that $Z\in\YY$.

  The other possibility is that $Z=X^G$ for some part $X\in\XX$.
  Now because $(\XX,\KK)$ is $G$-invariant,
  we know that $X$ is a union
  of $G$-orbits of $\Irr(N)$,
  so we can calculate by Lemma \ref{lem_calculatesigmaXG}
  that $\sigma_Z = \sigma_{(X^G)}=(\sigma_X)^G$.
  Because $N\norm G$,
  we know that $(\s_X)^G$ vanishes outside of $N$;
  therefore $\s_Z$ is constant on every set $\tw{L}\in\tw{\LL}-\{N\}$.
  Moreover,
  because the set $X$ is $G$-invariant, 
  the character $\s_X$ of $N$ is invariant under the action of $G$,
  so $(\s_Z)_N=((\s_X)^G)_N=\ix{G}{N} \s_X$.
  Because $\s_X$ is constant on every part $K\in\KK$,
  so too is $\s_Z$.

  Thus $\s_Z$ is constant on $M$ for every $Z\in\ZZ$ and every $M\in\MM$,
  so we conclude that $(\ZZ,\MM)$ is a supercharacter theory of $G$.\qed
\end{proof}

We call the supercharacter theory of $G$ 
constructed in the preceding theorem
the \defnstyle{$*$-product} of $(\XX,\KK)$ and $(\YY,\LL)$ and write it
$(\XX,\KK)*(\YY,\LL)$.

%

\section{Superinduction}

Before proceeding further, we need an analogue to induction
for supercharacters.
Let $H$ be a subgroup of $G$.
If $\CCC\in\Sup(G)$,
it may be that  $H$ has some supercharacter theory 
naturally related to $\CCC$.
For example, Diaconis and Isaacs used two-sided orbits to define
a standard supercharacter theory for algebra groups.
Thus the same construction yields supercharacter theories
both for an algebra group $G=1+\mathfrak{n}$ 
 and for an algebra subgroup $H=1+\mathfrak{m}$,
where $\mathfrak{m}\sseq\mathfrak{n}$.
For a supercharacter $\phi$ of $H$, however,
there is no guarantee that the induced character $\phi^G$ will be
a superclass function of $G$.
This is remedied in \cite{diaconis_isaacs}
by the invention of a ``superinduction''
for algebra groups,
which was further studied in \cite{thiem_marberg}.
We here generalize superinduction to arbitrary supercharacter theories.
\begin{definition}\label{defn_superinduction}
  Fix a supercharacter theory of a group $G$.
  Let $H$ be a subgroup of $G$ and let $\phi: H\to\C$ be a function.
  Then the \defnstyle{superinduced function} $\phi^{(G)}: G\to\C$ is defined by
  $$\phi^{(G)}(x) = \ix{G}{H} \rec{|[x]|} \sum_{y\in [x]} \phi^0 (y),$$
  where $[x]$ denotes the superclass containing $x$
  and $\phi^0(y)$ equals $\phi(y)$ if $y\in H$, but equals zero otherwise.
\end{definition}

Thus $\phi^{(G)}(x)$ is the average value of $\phi$ on the superclass containing $x$,
multiplied by $\ix{G}{H}$.  
It is clear that $\phi^{(G)}$ is therefore a superclass function,
although it need not be a character.
Note that Definition \ref{defn_superinduction}
agrees both with 
ordinary induction, if the superclasses are simply the conjugacy classes,
and with the definition of \cite{diaconis_isaacs}, 
if we are working with the standard supercharacter theory of algebra groups.
We will need the analogues for superinduction of two elementary results about ordinary induction.

\begin{lemma}[Frobenius Reciprocity]
  Let $H$ be a subgroup of a group $G$, and fix a supercharacter theory of $G$.
  Let $\phi$ be a class function of $H$ and $\theta$ a superclass function of $G$.
  Then $[\phi^{(G)},\theta] = [\phi,\theta_H]$.
\end{lemma}
\begin{proof}
  Let $\KK$ be the set of superclasses of $G$.  We calculate that
  \begin{eqnarray*}
    [\phi^{(G)},\theta] &=& \rec{|G|} \sum_{x\in G} \phi^{(G)}(x) \bar{\theta(x)} \\
                        &=& \rec{|G|} \sum_{x\in G} \ix{G}{H} \rec{|[x]|} \sum_{y\in[x]} \phi^0(y) \bar{\theta(x)} \\
                        &=& \rec{|H|} \sum_{x\in G} \rec{|[x]|} \sum_{y\in[x]} \phi^0(y) \bar{\theta(y)} \\
                        &=& \rec{|H|} \sum_{K\in\KK} \sum_{y\in K\cap H} \phi(y) \bar{\theta(y)} \\
                        &=& \rec{|H|} \sum_{y\in H} \phi(y) \bar{\theta(y)} \\
                        &=& [\phi,\theta_H].\qquad \Box
  \end{eqnarray*}
\end{proof}

In general $\phi^{(G)}$ need not be a character, even if $\phi$ is a character (see \cite{diaconis_isaacs} for an example).
If we start with a character $\chi$ of $G$, however, then for each supercharacter $\s_X$
$$[(\chi_N)^{(G)}, \s_X ] = [\chi_N, (\s_X)_N]$$
is a nonnegative integer; 
thus $(\chi_N)^{(G)}$ is indeed a character.

Just as normal subgroups are those subgroups which are unions of conjugacy classes,
so those subgroups which are unions of superclasses play an analogous role for superinduction.
Let $\CCC=(\XX,\KK)\in\Sup(G)$, which corresponds to the Schur ring $S_\KK$.
Suppose $N$ is a subgroup of $G$ that is a union of superclasses;
then we say $N$ is \defnstyle{$\CCC$-normal}.
Such an $N$ has also been called 
an \defnstyle{$S_\KK$-subgroup} \cite{muzychuk_ponomarenko, leung_man1996}
or a \defnstyle{supernormal} subgroup \cite{marberg_normality}.

\begin{lemma}\label{lem_superinducethenrestrict}
  Let $G$ be a group and $\CCC\in\Sup(G)$.
  Let $N\leq G$ be $\CCC$-normal,
  and let $\phi: N\to\C$ be a function
  constant on those superclasses that lie in $N$.
  Then $(\phi^{(G)})_N = \ix{G}{N} \phi$.
\end{lemma}
\begin{proof}
  Let $n\in N$.
  Then since $N$ is $\CCC$-normal, $\phi^0(y)=\phi(n)$ for all $y$ in the superclass of $n$.  Thus
  \begin{eqnarray*}
    \left(\phi^{(G)}\right)_N(n) = \ix{G}{N} \rec{|[n]|} \sum_{y\in [n]} \phi^0(y) 
                                 &=& \ix{G}{N} \rec{|[n]|} \sum_{y\in [n]} \phi(n) \\
                                 &=& \ix{G}{N} \phi(n).\qquad\Box
  \end{eqnarray*}
\end{proof}

We pause to note that
$\CCC$-normality can also be described in terms of characters.
Let $N\norm G$ and consider the supercharacter theory $\M(N)*\M(G/N)$,
where $\M(H)$ denotes the maximal supercharacter theory of a group $H$.
There are three superclasses, namely $\{1\}$, $N-\{1\}$, and $G-N$;
the corresponding partition of $\Irr(G)$
is $\{\{1_G\},\Irr(G/N)-\{1_G\},\Irr(G)\setminus\Irr(G/N)\}$.
For each $\CCC=(\XX,\KK)\in\Sup(G)$;
we see that $N$ is $\CCC$-normal 
if and only if $\CCC\pleq\M(N)*\M(G/N)$,
which is true if and only 
$\Irr(G/N)$ is the union of some members of $\XX$.\label{disc_altCnormalX}

\section{Restricting theories}

As noted in \cite[Section 3.2]{muzychuk_ponomarenko},
Schur rings of an $S_\KK$-subgroup $N$ and the quotient group $G/N$ can be defined
using the partitions
$\KK_N = \{K\in\KK : K\sseq N\}$
and $\KK/N=\{KN/N: K\in\KK\}$;
this latter is guaranteed to be a partition of $G/N$ by the following:
\begin{lemma}[\cite{leung_ma}, Lemma 1.2(ii)]\label{leungmalemma}
  Let $S_\KK$ be an S-ring over $G$ with Schur partition $\KK$,
  and let $N$ be a normal $S_\KK$-subgroup of $G$.
  If $K,L\in\KK$, then either $KN/N \cap LN/N=\emptyset$ or $KN/N=LN/N$.
\end{lemma}
We shall see that the corresponding supercharacter theories 
exhibit a nice symmetry between superclasses and supercharacters,
but first we need to develop a little notation.

If $Z\sseq\Irr(G)$,
let $f(Z)$ denote the set of all 
irreducible constituents of $(\s_Z)_N$.
Note that if $Z$ is a union of sets of the form $\Irr(G|\psi)$
for various $\psi\in\Irr(N)$,
then $(f(Z))^G=Z$.


\begin{lemma}\label{lem_fXsdisjoint}
  Let $G$ be a group, let $\CCC=(\XX,\KK)\in\Sup(G)$, and let $N$ be a $\CCC$-normal subgroup of $G$.
  Then for every $X,Y\in\XX$, either $f(X)\cap f(Y)=\emptyset$ or $f(X)=f(Y)$.
\end{lemma}
\begin{proof}
  Suppose $f(X)\cap f(Y)\neq\emptyset$;
  then $[(\s_X)_N, (\s_Y)_N ]>0$.
  By Frobenius reciprocity on superinduction, we then have that
  $[((\s_X)_N)^{(G)}, \s_Y ] >0$.
  But $((\s_X)_N)^{(G)}$ is a superclass function of $G$,
  and hence a linear combination of supercharacters.
  Since the supercharacters have disjoint supports,
  it follows that every irreducible constituent of $\s_Y$
  is also an irreducible constituent of $((\s_X)_N)^{(G)}$.
  
  Now let $\alpha\in f(Y)$. 
  Then there exists some $\beta\in\Irr(G|\alpha)$ that is an irreducible constituent of $\s_Y$,
  so $\beta$ is also an irreducible constituent of $((\s_X)_N)^{(G)}$.
  But then by Clifford theory, 
  $\alpha$ must be a constituent of
  $$\left(((\s_X)_N)^{(G)}\right)_N = \ix{G}{N} (\s_X)_N$$
  (by Lemma \ref{lem_superinducethenrestrict}).
  Thus $\alpha$ must be an irreducible constituent of $(\s_X)_N$
  and so $\alpha\in f(X)$.
  Hence $f(Y)\sseq f(X)$; 
  but by symmetry $f(X)\sseq f(Y)$ as well, completing the proof.\qed
\end{proof}
%

Because every irreducible character of $N$ lies in some $f(X)$,
it follows that $\{f(X): X\in\XX\}$ is a partition of $\Irr(N)$.
We are now ready to define the restrictions of $\CCC$ to normal subgroups and quotients.

\begin{definition}\label{defn_restrictdownandup}
  Let $G$ be a group, let $\CCC=(\XX,\KK)\in\Sup(G)$,
  and let $N$ be a $\CCC$-normal subgroup of $G$.
  Defining $f$ as above, 
  let
  $$\CCC_N = 
    \big( \{f(X): X\in\XX\}
           \widecomma
           \{K\in\KK: K\sseq N\} \big)$$
  and
  $$\CCC^{G/N} = 
    \big( \{X\in\XX: X\sseq\Irr(G/N)\}
           \widecomma
           \{KN/N: K\in\KK\} \big).$$
\end{definition}

To prove that these ordered pairs are in fact supercharacter theories,
we shall adjust $\CCC$ to a different theory slightly above it in the lattice $\Sup(G)$.
Let $\m(G/N)$ denote the minimal supercharacter theory in the lattice $\Sup(G/N)$,
namely the ordinary character theory of $G/N$,
and let $\mm{N}{G}$ denote the supercharacter theory $\Conj{G}{N}*\m(G/N)\in\Sup(G)$.
Then the superclasses of $\mm{N}{G}$ are those conjugacy classes of $G$ which lie in $N$,
together with the nontrivial conjugacy classes of $G/N$ pulled back to $G$.
In the corresponding partition of $\Irr(G)$, 
the characters in $\Irr(G/N)$ are in singleton parts,
while every part outside $\Irr(G/N)$
is of the form $\Irr(G|\psi)$ for some $\psi\in\Irr(N)$.

In the proof of the following proposition, let us say ``$\XX$ is constant on $\KK$'' to mean
that for each $X\in\XX$, the character $\s_X$ is constant on every part $K\in\KK$.

\begin{proposition}\label{lem_restrictworks}
  Let $N$ be a subgroup of a group $G$, 
  let $\CCC\in\Sup(G)$, and suppose $N$ is $\CCC$-normal.
  Then $\CCC_N$ is a $G$-invariant supercharacter theory of $N$
  and $\CCC^{G/N}$ is a supercharacter theory of $G/N$.
  Moreover, $$\CCC_N*\CCC^{G/N}=\CCC\join\mm{N}{G}.$$
\end{proposition}
\begin{proof}
  Let $\BBB=\CCC\join\mm{N}{G}$ 
  and consider $\BBB_N$ and $\BBB^{G/N}$.
  Note that replacing $\CCC$ with $\BBB$ 
  artificially fuses together the superclasses outside $N$
  in such a way that they become unions of $N$-cosets.
  However, since $KN/N$ and $LN/N$ are either disjoint or equal
  for superclasses $K$ of $\CCC$ by Lemma \ref{leungmalemma},
  it follows that the set $\{MN/N: M\in\MM\}$ is the same whether we take
  $\MM$ to be the superclasses of $\CCC$ or of $\BBB$.

  Replacing $\CCC$ with $\BBB$ similarly coarsens the partition of characters
  until every part outside $\Irr(G/N)$ is a union of sets of the form $\Irr(G|\psi)$.
  But since $f(X)$ and $f(Y)$ are either disjoint or equal by Lemma \ref{lem_fXsdisjoint},
  we obtain the same partition $\{f(Z): Z\in\ZZ\}$ of $\Irr(N)$
  whether we take the partition $\ZZ$ of $\Irr(G)$ from $\CCC$ or from $\BBB$.

  Replacing $\CCC$ with $\BBB$ does not change 
    the superclasses which lie within $N$, however,
    nor that portion of the partition of $\Irr(G)$ which lies within $\Irr(G/N)$.
  Consequently $N$ is $\BBB$-normal and $\BBB_N=\CCC_N$ and $\BBB^{G/N}=\CCC^{G/N}$.
  %
  Write $\BBB=(\ZZ,\MM)$, and let
  \begin{eqnarray*}
    \KK &=& \{M\in\MM: M\sseq N\},                                 \\
    \LL &=& \{MN/N: M\in\MM\}        \\ 
    \XX &=& \{f(Z): Z\in\ZZ\}, \mbox{~and}\\
    \YY &=& \{Z\in\ZZ: Z\sseq\Irr(G/N)\},                          
  \end{eqnarray*}
  so that $\BBB_N=\CCC_N = (\XX,\KK)$ and $\BBB^{G/N}=\CCC^{G/N}=(\YY,\LL)$.
  Note that since each part of $\MM$ outside $N$ is a union of $N$-cosets,
  we have $|\MM|=|\KK|+|\LL|-1$;
  likewise since each member of $\ZZ$ outside $\Irr(G/N)$ is a union of sets of the form $\Irr(G|\psi)$,
  we have $|\ZZ|=|\XX|+|\YY|-1$.
  
  Let us now verify that the sets $\XX$, $\KK$, $\YY$, and $\LL$
  are partitions of the appropriate sets.
  Since $N$ is $\BBB$-normal, it follows that $\KK=\{M\in\MM: M\sseq N\}$
  is a partition of $N$.
  That $\LL$ is a partition of $G/N$ follows from Lemma \ref{leungmalemma}.
  As for the characters,
  $\YY=\{Z\in\ZZ: Z\sseq\Irr(G/N)\}$
  is a partition of $\Irr(G/N)$
  because $N$ is $\BBB$-normal,
  and we saw above that the set 
  $\{f(Z): Z\in\ZZ\}$ is a partition of $\Irr(N)$.

  To show that $\BBB_N$ and $\BBB^{G/N}$ are indeed supercharacter theories,
  it remains to show that the purported supercharacters are constant on the superclasses,
  that $|\XX|=|\KK|$, and that $|\YY|=|\LL|$.
  
  Consider first $\BBB_N$. 
  Let $X\in\XX$.  
  If $X=\{1_N\}$, then $\s_X=1_N$ is trivially constant on all $K\in\KK$;
  otherwise, $X=f(Z)$ for some $Z\in\ZZ$ with $Z\not\sseq\Irr(G/N)$.
  Since $Z$ is a union of sets of the form $\Irr(G|\psi)$,
  we have $Z=X^G$.
  Then because
  $$(\s_Z)_N = \left(\s_{(X^G)}\right)_N = \left( (\s_X)^G\right)_N = \ix{G}{N} \s_X$$
  is constant on each $K\in\KK$, so too is $\s_X$.
  We conclude that $\XX$ is constant on $\KK$.
  
  Now consider $(\YY,\LL)$,
  and let $Y\in\YY\sseq\ZZ$.  
  Then $\s_Y$ is constant 
  on every superclass $M\in\MM$.
  But since $\s_Y$ has $N$ in its kernel,
  when viewed as a character of $G/N$
  it is constant on the images of those superclasses,
  namely the members of $\LL$.
  We conclude that $\YY$ is constant on $\LL$.
  
  Now because $\XX$ is constant on $\KK$
  and $\YY$ is constant on $\LL$,
  by \cite[Theorem 2.2]{diaconis_isaacs}
  we know that $|\XX|\leq |\KK|$ and $|\YY|\leq |\LL|$.
  But then
  $$|\KK|+|\YY|-1\leq |\KK|+|\LL|-1=|\MM|=|\ZZ|=|\YY|+|\XX|-1\leq |\YY|+|\KK|-1,$$
  so equality must hold throughout; hence $|\XX|=|\KK|$ and $|\YY|=|\LL|$.
  We conclude that $\CCC_N=\BBB_N=(\XX,\KK)$ is a supercharacter theory of $N$
  and $\CCC^{G/N}=\BBB^{G/N}=(\YY,\LL)$ is a supercharacter theory of $G/N$;
  the former is $G$-invariant 
  because its superclasses are also superclasses of $\CCC$.
  Finally, by definition
  \begin{eqnarray*}
    \CCC_N*\CCC^{G/N}
        = (\XX,\KK)*(\YY,\LL)
        &=& \left(\YY\cup \XX^G - \{\Irr(G/N)\} \widecomma \KK \cup \tw{\LL}-\{N\} \right) \\
        &=& (\ZZ, \MM) \\
        &=& \CCC\join\mm{N}{G},
  \end{eqnarray*}
  as desired.\qed
\end{proof}

\section{Wedge products}\label{sect_wtp}

In 1996 K.H. Leung and S.H. Man gave
a qualitative classification 
of the S-rings of cyclic groups \cite{leung_man1996}.
To do so they defined a ``wedge product'' of S-rings as follows.

\begin{proposition}[\cite{leung_man1998}, Proposition 1.4]
  Let $G$ be a group with subgroups $N$ and $M$
  such that $N\leq M$ and $N\norm G$.
  Let $\rho: G\to G/N$ be the natural surjection.
  For an S-ring $S_\KK$ over $M$ with Schur partition $\KK$, define
    $$\rho^*(S_\KK) = \spn_\C\{\hat{\rho(K)}: K\in\KK\} \sseq \C[G/N].$$
  Let $S_\LL$ be an S-ring over $G/N$ with basic sets $\LL$
  such that $M/N$ is an $S_\LL$-subgroup,
  and suppose both that $\hat{H}\in S_\KK$
               and that $\rho^*(S_\KK) = \C[M/N] \cap S_\LL$.
  Then there exists an S-ring $S$ over $G$ with Schur partition
  $$\KK \cup \left\{ \rho\inv(E): E \in \LL, E\not\sseq M/N \right\}.$$
  Moreover, $S\cap \C[M]=S_\KK$ and $\rho^*(S)=S_\LL$.
\end{proposition}

Leung and Man call this the \defnstyle{wedge product} of $S_\KK$ and $S_\LL$
and denote it by $S_\KK\wedge S_\LL$.
Note that this product is only defined when $\hat{H}\in S_\KK$ and $\rho^*(S_\KK) = \C[M/N] \cap S_{\LL}$.
Muzychuk and Ponomarenko refer to this as a \defnstyle{generalized wreath product} \cite{muzychuk_ponomarenko}.

By Proposition \ref{sringbijection}, there must be a corresponding construction of supercharacter theories,
which we now provide.
We are considering the situation
when $N\leq M$ are normal subgroups of $G$,
with ``overlapping'' supercharacter theories
$\CCC\in\Sup(M)$ and
$\DDD\in\Sup(G/N)$,
as in the following diagram:
$$\makebox[0pt][l]{$\underbrace{\phantom{1\leq N\leq M}}_{\CCC}$}1\leq \overbrace{N\leq M\leq G}^{\DDD}$$
Our construction will be a generalization of the $*$-product above, in which case $N$ and $M$ were equal.
In order for $\CCC$ and $\DDD$ to combine to form a supercharacter theory for $G$,
they must satisfy certain conditions.
We will of course want $N$ to be $\CCC$-normal and $M/N$ to be $\DDD$-normal,
but we will also want the ``overlap'' of the two theories on $M/N$ to be the same;
more explicitly, we will require $\CCC^{M/N}=\DDD_{M/N}$.

\begin{theorem}\label{thm_wtpproduct}
  Let $G$ be a group with normal subgroups $N\leq M$.
  Suppose $\CCC\in\Sup_G(M)$ and $\DDD\in\Sup(G/N)$ such that
  \begin{enumerate}
    \item $N$ is $\CCC$-normal,
    \item $M/N$ is $\DDD$-normal, and
    \item $\CCC^{M/N}=\DDD_{M/N}$.
  \end{enumerate}
  Then there exists a unique supercharacter theory $\EEE\in\Sup(G)$
  such that $\EEE_M=\CCC$ and $\EEE^{G/N}=\DDD$
  and every superclass outside $M$ is a union of $N$-cosets.

  Using our earlier notation,
  if $\CCC=(\XX,\KK)$ and $\DDD=(\YY,\LL)$, 
  then
  \begin{equation*}
    \EEE=\left( \YY\cup\{X^G: X\in\XX,~X\not\sseq\Irr(M/N)\}
                \widecomma 
                \KK\cup\{\tw{L}: L\in\LL,~L\not\sseq M/N\}\right).
  \end{equation*}
\end{theorem}
\begin{proof}
  For every superclass $L$ of $\DDD$ lying outside $M/N$, 
  take its preimage $\tw{L}$ in $G$;
  because $M/N$ is $\DDD$-normal,
  this gives a partition of $G\setminus M$.
  To this set add all the superclasses of $\CCC$;
  since these partition $M$,
  the resulting set $\KK\cup \{\tw{L}: L\in\LL,~L\not\sseq M/N\}$
  is a partition of $G$
  which we shall call $\JJ$.
  Recalling that $|\CCC|$ denotes the number of superclasses of $\CCC$,
  note that $|\JJ|=|\CCC|+\left(|\DDD|-|\DDD_{M/N}|\right)$.
  
  Now because $N$ is $\CCC$-normal,
  the subset $\Irr(M/N)$ is a union of parts of $\XX$,
  as discussed in Section \ref{disc_altCnormalX}.
  Hence $\{X\in\XX: X\not\sseq\Irr(M/N)\}$ partitions $\Irr(M)-\Irr(M/N)$,
  so the set $\{X^G: X\in\XX,~X\not\sseq\Irr(M/N)\}$ partitions $\Irr(G)-\Irr(G/N)$
  since $\CCC$ is $G$-invariant.
  Since $\YY$ is a partition of $\Irr(G/N)$,
  the union $\YY\cup\{X^G: X\in\XX,~X\not\sseq\Irr(M/N)\}$
  is a partition of $\Irr(G)$; call it $\WW$.
  Note that 
  $$|\WW|=|\DDD|+(|\CCC|-|\CCC^{M/N}|)=|\CCC|+\left(|\DDD|-|\DDD_{M/N}|\right)=|\JJ|.$$
  Then to prove that $(\WW,\JJ)$ is a supercharacter theory of $G$,
  it remains only to show that $\s_W$ is constant on $J$ 
  for each $W\in\WW$ and each $J\in\JJ$.
  
  Let $W\in\WW$.  It may be that $W\in\YY$, 
  so that $\s_W$ is a supercharacter of $\DDD$.
  In this case, there are two sorts of sets $J\in\JJ$ to consider:
  those that lie in $G\setminus M$ and those within $M$.
  First, the supercharacter $\s_W$ of $\DDD$
  is constant on each superclass $L$ of $\DDD$ lying outside $M/N$,
  so $\s_W$ (viewed as a character of $G$) 
  is constant on each preimage $\tw{L}$ in $G\setminus M$.
  Thus $\s_W$ is constant on each set $J\in\JJ$ that lies in $G\setminus M$.
  Next, note that
  $\s_W$ is constant on the nontrivial superclasses of $\DDD_{M/N}=\CCC^{M/N}$,
  so $\s_W$ is also constant on the superclasses of $\CCC$.
  Hence $\s_W$ is constant on every set $J\in\JJ$ that lies in $M$.
  Therefore $\s_W$ is constant on every member of $\JJ$,
  under the supposition that $W\in\YY$.
  
  The other possibility is that $W=X^G$ 
  for some part $X\in\XX$ not lying in $\Irr(M/N)$.
  Since $\CCC$ is $G$-invariant, the set $X$ must be a union of $G$-orbits,
  so $\s_W=\s_{(X^G)}=(\s_X)^G$ by Lemma \ref{lem_calculatesigmaXG}.
  Then because $M\norm G$, we know that $\s_W$ vanishes outside $M$,
  and hence is constant on all parts $J\in\JJ$ that lie outside $M$.
  On the other hand, each part $J\in\JJ$ that lies in $M$ 
  is a superclass of $\CCC$,
  and when $\s_W$ is restricted to $M$, 
  the character $(\s_W)_M=\left((\s_X)^G\right)_M=\ix{G}{M}\s_X$ is constant
  on $J$ because $\s_X$ is.
  
  Hence $\s_W$ is constant on each part $J\in\JJ$ for all parts $W\in\WW$, 
  and we conclude that $(\WW,\JJ)$ is a supercharacter theory of $G$.  
  Let $\EEE=(\WW,\JJ)$; we need to show 
  that $\EEE$ satisfies the conclusions of the theorem.
  By construction, the superclasses of $\EEE$ that lie in $M$
  are the superclasses of $\CCC$, so $\CCC=\EEE_M$.
  Likewise the supercharacters of $\EEE^{G/N}$
  are those supercharacters of $\EEE$ that have $N$ in their kernels,
  namely the supercharacters of $\DDD$;
  hence $\DDD=\EEE^{G/N}$.
  Third, by construction
  the superclasses of $\EEE$ outside $M$ are preimages
  of certain superclasses of $\DDD$, so they are unions of $N$-cosets.
  
  Finally, to show uniqueness, suppose $\FFF\in\Sup(G)$
  satisfies the conditions that \mbox{$\FFF_M=\CCC$},
  that $\FFF^{G/N}=\DDD$,
  and that every superclass of $\FFF$ outside $M$
  is a union of $N$-cosets.
  Then $\FFF_M=\CCC=\EEE_M$, 
  so $\EEE$ has the same superclasses within $M$ as does $\FFF$.
  Moreover, because the superclasses of $\FFF$ outside of $M$ are unions
  of $N$-cosets, the set
  $$\begin{array}{r@{}c@{}l}
    \multicolumn{3}{l}{\{\mbox{superclasses of $\FFF$ outside $M$}\}} \\
     \qquad\qquad
      &{}={}& \{\mbox{superclasses of $\FFF\join\mm{N}{G}$ outside $M$}\} \\
      &=& \{\mbox{superclasses of $\FFF_N*\FFF^{G/N}$ outside $M$}\} \\
      &=& \{\mbox{superclasses of $\FFF_N*\DDD$ outside $M$}\} \\
      &=& \{\mbox{preimages of the superclasses of $\DDD$ outside $M/N$}\} \\
      &=& \{\mbox{superclasses of $\EEE$ outside $M$}\}.
  \end{array}$$
  Therefore $\FFF$ has the same superclasses as $\EEE$,
  so $\FFF=\EEE$ as desired.\qed
\end{proof}

Because the wedge-product notation ``$\CCC\wedge\DDD$'' of Leung and Man
might be confusing in the context of our lattice-theoretic $\join$ operation,
we will denote this product as $\CCC\wtp\DDD$.
It is easy to recognize and factor such products; the following is a reformulation
of \cite[Proposition 1.3]{leung_man1996} in the case of supercharacter theories.
\begin{proposition}\label{prop_wtprecognition}
  Let $G$ be a group, let $\CCC\in\Sup(G)$,
  and let $N$ and $M$ be $\CCC$-normal subgroups of $G$
  with $N\leq M$.
  Then $\CCC$ is a $\bigwtp$-product over $N$ and $M$
  if and only if every superclass outside $M$ is a union of $N$-cosets.
  In this case, $\CCC=\CCC_M\wtp\CCC^{G/N}$.
\end{proposition}

\section{Direct products}

We close by considering a much simpler construction.
Tamaschke proved that if $S_\KK$ and $S_\LL$ are Schur rings on $M$ and $N$, respectively,
then there is an S-ring of $G=M\by N$
with partition 
$$\left\{ \{ (m,n): m\in K, n\in L\}: K\in\KK, L\in\LL \right\}$$
\cite[Theorem 6.1]{tamaschke1970}.
This Schur ring is ring-isomorphic to $S_\KK \otimes S_\LL$;
Leung and Man refer to it as the dot product ``$S_M \cdot S_N$.''

The corresponding supercharacter theory is equally straightforward.
Given two supercharacter theories $(\XX,\KK)\in\Sup(M)$ and $(\YY,\LL)\in\Sup(N)$,
we shall form a ``product'' theory $(\XX,\KK)\by(\YY,\LL)$.
As Tamaschke did, let
$$\MM = \{ K\by L : K\in\KK,~L\in\LL \} \mbox{~where~}
       K\by L = \{ (m,n) : m\in K,~n\in L\}\sseq G.$$
On the character side, 
we know that $\Irr(G) = \Irr(M)\by\Irr(N)$,
so let
$$\ZZ = \{ X\by Y : X\in\XX,~Y\in\YY\} \mbox{~where~}
        X\by Y = \{ \phi\by\theta : \phi\in X,~\theta\in Y\}\sseq\Irr(G).$$

\begin{proposition}
Using the above notation, $(\ZZ,\MM)\in\Sup(G)$.
\end{proposition}
\begin{proof}
Certainly $|\ZZ| = |\XX||\YY| = |\KK| |\LL| = |\MM|$.
%
So it suffices to show for all sets $X\in\XX$, $Y\in\YY$, $K\in\KK$, and $L\in\LL$
that the character $\sigma_{X\by Y}$ 
is constant on the set $K\by L$.
For all $m\in M$ and $n\in N$, we have
\begin{eqnarray*}
\sigma_{X\by Y}((m,n))
    &=& \sum_{\phi\in X} \sum_{\theta\in Y} (\phi\by\theta)((1,1))\cdot (\phi\by\theta)((m,n)) \\
    &=& \sum_{\phi\in X} \sum_{\theta\in Y} \phi(1)\theta(1)\phi(m)\theta(n) \\
    &=& \sum_{\phi\in X} \phi(1)\phi(m) \sum_{\theta\in Y} \theta(1)\theta(n) \\
    &=& \sigma_X(m) \sigma_Y(n).
\end{eqnarray*}
Thus for all $m,m'\in K$ and all $n,n'\in L$,
$$\sigma_{X\by Y}((m,n))
   = \sigma_X(m)\sigma_Y(n)
   = \sigma_X(m')\sigma_Y(n')
   = \sigma_{X\by Y}((m',n')).$$
Thus $\sigma_{X\by Y}$ is in fact constant on $K\by L$.
We conclude that $(\ZZ,\KK)$ is indeed a supercharacter theory of $G$.\qed
\end{proof}

\bibliography{ssr_bib}

\end{document}